\documentclass[12pt,a4paper]{amsart}
\usepackage{enumerate}
\def\ham{\mathbb H_M}
\def\hamp{\mathbb H_{+}}
\def\hamm{\mathbb H_{-}}
\def\hams{{\mathcal H}_s}

\def\hama{\mathbb H_{M^\alpha}}
\newcommand{\cupdot}{\mathbin{\mathaccent\cdot\cup}}
\newcommand{\twoline}[2]{\genfrac{}{}{0pt}{}{#1}{#2}}
\newtheorem{theorem}{Theorem}
\newtheorem{lemma}[theorem]{Lemma}
\newtheorem{corollary}[theorem]{Corollary}
\newtheorem*{corollary*}{Corollary}
\newtheorem{definition}[theorem]{Definition}
\newtheorem*{definition*}{Definition}
\newtheorem*{theorem*}{Theorem}

\author[M. Paluszynski]{Maciej Paluszynski}
\address{\small Instytut Matematyczny, Uniwersytet Wroc\l awski,
 pl. Grunwaldzki 2/4, \hbox{50-384} Wroc\-\l aw, Poland}
\email{mpal@math.uni.wroc.pl}

\author[J. Zienkiewicz]{Jacek Zienkiewicz}
\address{\small
 Instytut Matematyczny, Uniwersytet Wroc\l awski,
 pl. Grunwaldzki 2/4, \hbox{50-384} Wroc\-\l aw, Poland}
\email{zenek@math.uni.wroc.pl}

\title[ ]{ Weak type (1,1) estimates for inverses  of discrete rough singular integral operators.}
\begin{document}
\begin{abstract} We obtain  weak type (1,1) estimates for the  inverses
of truncated discrete  rough Hilbert transform.  We include an example showing that our result is sharp.
One of the ingredients of the proof are regularity estimates for convolutions of singular measure
associated with the sequence $[m^\alpha]$, see \cite {UZ}.

\end{abstract}

\date{\today}
\thanks{The   second named author was supported by the NCN grant   UMO-2014/15/B/ST1/00060.  }
 \maketitle
\baselineskip=18pt

\section{Introduction} Suppose $1<\alpha\le1+\frac{1}{1000}$, $0<\theta<1$ are fixed parameters.  For a non-negative number $M$ we consider a family of operators on $\ell^2(\mathbb Z)$
\begin{equation}\label{ham_def}
\begin{aligned}
&\ham f(x)=\sum_{\twoline{M^\theta\le s \le M}{s-\text{dyadic}}}\hams f(x)=\\
&=\sum_{\twoline{M^\theta\le s \le M}{s-\text{dyadic}}}
\sum_{m>0}\varphi_s \Big(\frac {m^\alpha}{s}\Big) \frac{f(x-[m^\alpha])-f(x+[m^\alpha])}{m}, \quad x\in\mathbb Z
\end{aligned}
\end{equation}
for some sequence $\varphi_s$ which is uniformly in $C_c^\infty(\frac{1}{2},2)$. It is by now a
routine fact that the operators $\ham$, the truncated Hilbert transforms, are bounded on $\ell^p$,
$1<p<\infty$ with norm estimates uniform in $M$ and $\theta$. The analogous weak type $(1,1)$
estimate seems to be unknown. For a fixed $\theta$, by a rather routine application of the methods
of \cite{C1}, \cite{S} and \cite{UZ} the operators $\ham$ can be shown to be of weak type (1,1)
uniformly in $M$. The subject of the current paper has been inspired by \cite{C}. There, a theorem
has been proved (\cite{C}, Theorem 3), which for our purposes can be formulated as follows:
\begin{theorem*} Suppose $K$ is a kernel in $\mathbb R^d$ satisfying $K(x)=\Omega(x)/|x|^d$, where $\Omega$ is homogeneous of degree 0, $\Omega\in L^q(S^{d-1})$ and has mean 0. Denote $Kf=K*f$. Suppose further that for some $\lambda\in\mathbb C$ the operator $\lambda\,\rm{Id}+K$ is invertible in $L^2(\mathbb R^d)$. Then $(\lambda\,\rm{Id}+K)^{-1}$ is of form $\Lambda\,\rm{Id}+K'$, where the kernel $K'$ satisfies the same assumptions as $K$.
\end{theorem*}
It immediately implies:
\begin{corollary*}[\cite{C}, \cite{C1}, \cite{CR}, \cite{S1}]
In the setting of the above theorem, the operator $(\lambda\,\rm{Id} +K)^{-1}$ is of weak type (1,1).
\end{corollary*}
The principal object of  the current work is to extend the above theorem to the case of discrete
rough Hilbert transforms $\ham$.  For a fixed $\theta$ we prove    the uniform in $M$ estimates for
$\|(\lambda\,\rm{Id}+\ham)^{-1}\|_{\ell^1\to\ell^{1,\infty}}$, provided such an estimate exists in
the sense of $\ell^2$. By the previous general remark, this goal is accomplished through the
following representation theorem, which is the main result of this paper

\begin{theorem}\label{main_thm}
Suppose $1<\alpha\le1+\frac{1}{1000}$ and let $\theta $ be such, that $\alpha-1<\theta<1$.
 Fix $\lambda\in\mathbb C$ and suppose that for some constant $C_I$ we have
\begin{equation}\label{main_op}
\Big\|\big(\lambda\,\rm{Id}+\ham\big)^{-1}\Big\|_{\ell^2\to\ell^2}\le C_I, \quad\text{for }M\ge M_0.
\end{equation}
Then, there exists $M_1=M_1(C_I,\lambda)$ such that for $M\ge M_1$ the kernel of the operator $(\lambda\,\rm{Id}+\ham)^{-1}$ has the form
\begin{equation}\label{main_decomp}
\lambda_I\,\textup{Id}+\beta_I\,\ham+K,
\end{equation}
where $K$ is the classical discrete Calder\'on-Zygmund kernel, and we have a uniform in $M\ge M_1$ estimate
\begin{equation}\label{main_est}
\big|\lambda_I\big|+\big|\beta_I\big|+\|K\|_{\ell^2\to\ell^2}+\|K\|_{CZ} \le C_1(C_I,\lambda),
\end{equation}
where
\begin{equation*}
\|K\|_{CZ}=\sup_y\sum_{|x|\ge2|y|}|K(x-y)-K(x)|.
\end{equation*}
Moreover, the above restriction on $\theta$ is sharp (we make this statement    precise in   Theorem \ref{4_ex}
in the
next section).
\end{theorem}

Applying standard Banach algebras arguments (eg. \cite{GRS}), for each fixed $M$, the kernel of the
operator $(\lambda \,\textup{Id}+\,\ham)^{-1}$ is in $\ell^1((1+|x|)^N)$ for any $N\ge 0$. In
particular  $(\lambda \,\textup{Id}+\,\ham)^{-1}$ is bounded on $\ell^1$, but the weak type $(1,1)$
estimate obtained in this way becomes unbounded when $M\rightarrow \infty$. Also, by selfduality of
the multiplier problem, the uniform in $M$ upper bound for
$\|(\lambda\,\rm{Id}+\ham)^{-1}\|_{\ell^1\to\ell^{1,\infty}}$ requires assumption \eqref{main_op}.

It is worthwhile to put our result in a more general context. First we note that for the
convolution Calder\'on-Zygmund operators in the continuous setting, the invertibility theorems are
by now classical. Similarly, the resolvent of the discrete Hilbert transform, if it exists  as an
operator on $\ell^2(\mathbb{Z})$, is  a  discrete Calder\'on-Zygmund operator. This fact seems to
be folklore and can be proved  by an application of Fourier transform or by the method of \cite{C}.
The discrete analogues of the classical singular integrals have been studied intensively, see some
examples \cite {B}, \cite{BM}, \cite{C2}, \cite{IW}, \cite{LV} \cite{MSW}. We believe, that our
results fit well within this line of research.

\vskip.3cm {\bf Acknowledgement.} We thank  the reviewer for the remarks which significantly
improved the overall presentation of the paper.

\section{Main Theorem} Let us recall, that we have fixed parameters $\alpha,\ \theta$ with $1<\alpha \le 1+\frac{1}{1000}$,
$0<\theta <1$. We introduce a family of algebras, which are subalgebras of the algebra of operators on $\ell^2$.
\begin{definition}\label{1}
We consider the family of operators $T$, which are convolution operators on $\mathbb Z$, with kernels of the form
\begin{equation}\label{2}
T=\lambda\,\textup{Id}+\beta\,\ham+\sum_{\twoline{M^\theta\le s<\infty}{s-\text{dyadic}}} K_s,
\end{equation}
(we  identify convolution operator with its kernel), where the operator $\ham$ is the truncated Hilbert transform:
\begin{equation}\label{hilb_def1}
\ham f(x)=\sum_{\twoline{M^\theta\le s \le M}{s-\text{dyadic}}} \hams f(x)
\end{equation}
with
\begin{equation}\label{hilb_def2}
\hams f(x)=  \sum_{m>0}\varphi_s \Big(\frac{m^\alpha}{s}\Big) \frac{f(x-[m^\alpha])-f(x+[m^\alpha])}{m}
\end{equation}
for some sequence $\varphi_s$ which is uniformly in $C_c^\infty(\frac{1}{2},2)$.
We require that the kernels $K_s$ satisfy:
\begin{enumerate}[$(i)_s$]
\item  $\sum_x K_s(x)=0$,
\item $\textup{supp }K_s\subset[-s,s]$,
\item $\sum_x|K_s(x)|^2\le\frac{D^2_s}{s}$,
\item $\sum_x|K_s(x+h)-K_s(x)|^2\le\frac{D^2_s}{s}\Big(\frac{|h|}{s}\Big)^{\gamma_0}$,
\end{enumerate}
for some small positive $\gamma_0$ depending only on $\delta=\theta-(\alpha-1)$.

For a fixed $M$ we put
\begin{equation*}
\|\{K_s\}\|_{A_M}=\sup_{\twoline{M^\theta\le s<\infty}{s-\text{dyadic}}}D_s,
\end{equation*}
and
\begin{equation}\label{3}
\|T\|_{A_M}=\inf\{|\lambda|+|\beta|+\|\{K_s\}\|_{A_M}\},
\end{equation}
where the infimum is taken over all representations of the operator $T$ in the form \eqref{2}.
\end{definition}

 In fact $A_M$ is a Banach algebra with the norm $C\|T\|_{A_M}$ for certain constant $C$ independent of $M$.
Moreover,
\begin{equation*}
K=\sum_{\twoline{M^\theta\le s<\infty}{s-\text{dyadic}}}K_s
\end{equation*}
is  Calder\'on-Zygmund kernel with constant controlled by  $\|T\|_{A_M}$.

We are now ready to formulate
the two theorems leading immediately to Theorem \ref{main_thm}.

\begin{theorem}\label{4}
Let  $\theta>\alpha-1$.
Assume that for some fixed $\lambda\in {\mathbb C}$ and a constant $C_I$ all operators $\lambda\,\rm{Id}+\ham$ are invertible
for $M\ge M_0$ and $\|(\lambda\,\rm{Id}+\beta\ham)^{-1}\|_{\ell^2\to\ell^2}\le C_I$. Then for $M\ge M_1$
we have $\|(\lambda\,\rm{Id}+\beta\,\ham)^{-1}\|_{A_M}\le C(C_I,\lambda)$.
\end{theorem}
\begin{theorem}\label{4_ex}
Let  $\theta<\alpha-1$. There exists a sequence of functions $\varphi_s$  and a compact set $\Gamma
\subset \mathbb C$
 such that the corresponding Hilbert transform \eqref{hilb_def2} satisfies
$\|(\lambda\,\rm{Id}+\ham)^{-1}\|_{\ell^2\to\ell^2}\le C_I$ for all $M$ and $\lambda \in \Gamma$,
and the estimate $\|(\lambda\,\rm{Id}+\,\ham)^{-1}\|_{\ell^1 \rightarrow \ell^{1,\infty}}\le C $,
does not, for any $C$, hold uniformly in $\lambda \in \Gamma$ and $M$.
\end{theorem}

 {\bf Remarks:}
\newline
(i) The range of $\alpha$'s considered in Theorem \ref{4} is not optimal, and can be improved using
the methods from \cite{M}, \cite{UZ} or a variant of the argument used in this work to prove Lemma
\ref{war_reg}.
\newline (ii) Theorem \ref{4}
is probably also true with $[m^\alpha]$ replaced by $[m^\alpha\varphi(m)]$, where $\varphi$ is a
function of the Hardy class considered in \cite{M}.
\newline (iii) For  values of $\theta<1$ close to $1$  Theorem \ref{4} could be proved
 using  regularising effect in $\ell^2$ of the kernel $\mathbb H_M$. Known estimates for  the Fourier transform $\hat{\mathbb H}_M$  seem, however,
to be too weak to cover the entire range of $\theta$ considered in this paper.
\newline (iv) In the proof of Lemma \ref{1:8} we could have used a weaker statement of Lemma
\ref{war_reg}, at a cost of a more sophisticated argument.
We believe that Lemma \ref{war_reg} is
of some independent interest, because of its relation to certain type of Waring problem (see \cite{D}, \cite{Se}).
This is one reason we have chosen the variant of proof we present.
\newline (v) Condition \eqref{main_op} is always satisfied for sufficiently large $|\lambda|$. If we only consider real valued $\varphi_s$, more can be said. Since the kernels $\ham$ are anti-symmetric, the Fourier transform $\widehat\ham$ is purely imaginary and also anti-symmetric. Thus \eqref{main_op} is equivalent to $\lambda\notin[-i\,N,i\,N]$, where $N\ge0$. Using the estimates from \cite{D} it can be shown that
\begin{equation*}
N=\limsup_{M\to\infty}\,\sup_{\xi\in\mathbb R}\left|c_\alpha\sum_{\twoline{M^\theta\le s\le M}{s-\text{ dyadic}}}\int_0^\infty\sin(\xi ts^\alpha)\,\varphi_s(t^{1/\alpha})\frac{dt}{t^{1-1/\alpha}}\right|
\end{equation*}
(where $c_\alpha$ is explicitly computable).
\newline (vi) We refer the reader to our subsequent paper \cite{PZ} for a sharper version of Theorem
\ref{4_ex}, see Remark at  the end of Section 5.

Theorem \ref{4} is an immediate consequence of the following result, which exploits the mixed-norm
submultiplicity properties of algebras $A_M$. The idea of using such estimates to solve the problem
of invertibility of singular integral operators first appeared in \cite{C} .

\begin{theorem}\label{5}
Let $A_M$, $M\ge M_0\ge1$ be a family of algebras, consisting of bounded convolution operators on $\ell^2$, with norms $\|\cdot\|_{A_M}$, satisfying
\begin{align}
\|T_1\,T_2\|_{A_M}&\le C_A\big(\|T_1\|_{\ell^2\to\ell^2}\|T_2\|_{A_M}+\|T_1\|_{A_M}\|T_2\|_{\ell^2\to\ell^2}\big)+\label{6}\\
&\qquad\qquad+C_A\epsilon(M)\|T_1\|_{A_M}\|T_2\|_{A_M},\nonumber\\
\|T_1\,T_2\|_{A_M}&\le C_A\|T_1\|_{A_M}\|T_2\|_{A_M},\label{7}
\end{align}
where the constant $C_A$ does not depend on $M$ and $\epsilon(M)\rightarrow 0$ as $M\rightarrow \infty$.
Suppose all operators from the sequence $T^{(M)}$ are  invertible on $\ell^2$ and  satisfy:
\begin{equation}\label{8}
\begin{aligned}
 \|(T^{(M)})^{-1}\|_{\ell^2\to\ell^2}+\|T^{(M)}\|_{A_M}&\le K\qquad\text{$K$ independent of $M\ge M_0$},\\
\|T^{(M)}\|_{\ell^2\to\ell^2}&\le\delta<1.
\end{aligned}
\end{equation}
Then for an $M_1\ge M_0$, sufficiently large and depending only on $K$ and $\delta$, and all $M\ge M_1$,
$T^{(M)}$ are invertible in $A_M$, with
\begin{equation*}
\big\|(T^{(M)})^{-1}\big\|_{A_M}\le C=C(K,\delta),
\end{equation*}
with $C(K,\delta)$ independent of $M\ge M_1$.
\end{theorem}
\begin{proof}[Proof]
We will drop the superscript $M$ and denote $T^{(M)}$ by $T$.
We first prove that
there exist constants $C,N_0$ and $\delta_1<1$, depending only on $K,\delta,C_A$, such that
\begin{equation}\label{9}
\big\|T^n\big\|_{A_M}\le C\,\delta_1^n,\quad n\ge N_0.
\end{equation}
A simple inductive argument shows an estimate
\begin{align*}
\big\|T^{2^N}\big\|_{A_M}&\le 2^N\,C_A^N\big\|T^{2^{N-1}}\big\|_{\ell^2\to\ell^2}\dots\|T\|_{\ell^2\to\ell^2}\|T\|_{A_M}+\\
&\qquad+\epsilon \,G_N\big(\|T\|_{A_M},\|T\|_{\ell^2\to\ell^2}\big),
\end{align*}
where $G_N$ is a polynomial of degree $\le2^N$, with non-negative coefficients. Suppose an operator $T$ satisfies \eqref{8}. Then, clearly
\begin{equation*}
\big\|T^{2^N}\big\|_{A_M}\le (2\,C_A)^N\,\delta^{2^N-1}\,K+\epsilon\,G_N(K,\delta).
\end{equation*}
Choose $N_0$ such, that
\begin{equation*}
(2\,C_A)^{N_0}\,\delta^{N_0}\,K\le\frac{1}{4\,C_A},
\end{equation*}
and $M_1\ge M_0$ so that also
\begin{equation*}
\epsilon(M)\,G_{N_0}(K,\delta)\le\frac{1}{4\,C_A},\qquad M\ge M_1.
\end{equation*}
We get
\begin{equation*}
\big\|T^{2^{N_0}}\big\|_{A_M}\le\frac{1}{2\,C_A},\qquad M\ge M_1.
\end{equation*}
By \eqref{7} and  a standard Banach algebras considerations we get
\begin{equation}\label{banach_1}
\big\|T^n\big\|_{A_M}\le\Big(\frac{1}{2}\Big)^{\frac{n}{2^{N_0}}}\,\cdot\,C_{C_A,K,\delta}.
\end{equation}

Suppose that the positive  invertible on $\ell^2$ operator $T$ satisfies \eqref{8}. Then
$\delta\le I-T\le 1-K^{-1}$ so $I-T$ satisfies \eqref{8}. Applying \eqref{banach_1}
to the Neumann series representation of $T^{-1}$
we get an estimate
$\|T^{-1}\|_{A_M}\le C_{K,\delta,C_{A}}$.

Now, if $T$ is an arbitrary operator, invertible on $\ell^2$ and satisfying \eqref{8}, we apply the
above conclusion to $T^*\,T$ and $T\,T^*$ and the proof of Theorem \ref{5} is concluded.
\end{proof}
The fact that the algebra norms $\|\cdot\|_{A_{M}}$ satisfy the hypotheses \eqref{6} and \eqref{7}
will follow from a series of  lemmas, which are gathered in the next section.

\section{Lemmas}
In this section we fix $\theta=\alpha-1+\delta$, $\delta>0$.
Let $\varphi\in C_c^\infty(\frac12 ,2)$, and, for convenience let us introduce an operator $H_s$:
\begin{equation}\label{H_s_def}
H_sf(x)=\mathcal H_{s^{\alpha}}f(x)=\sum_{m>0}\varphi\Big(\frac {m}{s}\Big)
\frac{f(x-[m^\alpha])-f(x+[m^\alpha])}{m},
\end{equation}
where $\hams$ corresponds to the functions $\tilde\varphi_s(t)=\varphi(t^{1/\alpha})$. Let us denote by $H_s(x)$ the kernel of this operator.
\begin{lemma}{\label{war_reg}}
Fix $1<\alpha<1+\frac{1}{1000}$ and $\delta_L>0$. Then there exist functions $G_s(x)$, $E_s(x)$ and
an exponent $\gamma(\delta_L)$ independent of $s$, such that
\begin{equation}\label{conv_dec}
 H_s*H_s(x)= G_s(x)+E_s(x)+\frac{C}{s}\,\delta_0(x)
\end{equation}
where
\begin{equation}\label{G_sup}
 |G_s(x)|+|E_s(x)|\le Cs^{-\alpha},\qquad\rm{supp }\ E_s\subset [-s^{\alpha-1+\delta_L},s^{\alpha-1+\delta_L}]
\end{equation}
and
\begin{equation}\label{G_reg}
|G_s(x+u)-G_s(x)|\le  Cs^{-\alpha} |u\,s^{-\alpha}|^{-\gamma(\delta_L)}
\end{equation}
where the constants $C$ depends only on $\varphi$.
\end{lemma}
This lemma is the main technical tool we use. We postpone its proof to the next section. In this
section we will apply this lemma to $\hams$, that is with $s$ replaced by $s^{\frac1\alpha}$.

\begin{lemma}\label{1:3}
 Let $\psi\in\text{\rm C}_c^\infty(\mathbb R)$, $\psi\equiv1$ for $|x|\le1$, $\psi\equiv0$ for $|x|\ge2$. For a given convolution kernel $K$ on $\mathbb Z$ we define truncated kernels:
\begin{equation*}
K_R(x)=K(x)\cdot\psi\Big(\frac{x}{R}\Big),
\end{equation*}
Then for $R\ge1$ we have
\begin{equation*}
\|K_R\|_{\ell^2\to\ell^2}\le C\,\|K\|_{\ell^2\to\ell^2},
\end{equation*}
where the constant $C$ is independent of $R$.
\begin{proof}
This is immediate by Fourier transform.
\end{proof}
\end{lemma}
\begin{lemma}\label{1:5}
For an operator $T$ as in \eqref{2}, we have
\begin{equation*}
|\lambda|\le\|T\|_{\ell^2\to\ell^2}+\epsilon(M)\|T\|_{A_M}.
\end{equation*}
\end{lemma}
\begin{proof}
It suffices to observe, that
\begin{equation*}
<\ham\,\delta_0,\delta_0>=0,
\end{equation*}
and by $(iii)$ of definition \ref{1}
\begin{equation*}
|K_s(0)|^2\le\frac{\|T\|^2_{A_M}}{s}.
\end{equation*}
Then, for  $\epsilon(M)\le C\,M^{-\theta/2}$  the conclusion follows from
\begin{equation*}
\lambda=<T\,\delta_0,\delta_0>-\sum_{\twoline{M^\theta\le s<\infty}{s-\text{dyadic}}}K_s(0).
\end{equation*}
\end{proof}

\begin{lemma}\label{1:4}
Let $T$ be the kernel of the form  \eqref{2}. Then $T$ admits a representation
\begin{equation*}
\lambda\,\textup{Id}+\beta\,\sum_{\twoline{M^\theta\le s\le M}{s-\text{dyadic}}}\hams+\sum_{\twoline{M^\theta\le s<\infty}{s-\text{dyadic}}}K'_s,
\end{equation*}
where:
\begin{equation*}
\hams(x)=\Big(\psi\big(\frac{x}{2s}\big)-\psi\big(\frac{x}{2s}  \big)\Big)\,\ham(x),\qquad s\ge M^\theta,\text{ dyadic},
\end{equation*}
the function $\psi$ is the same smooth cutoff function as in the previous lemma, the kernels $K'_s$ satisfy  conditions
${ (i)_{s}...(iv)_{s}}$
from Definition \ref{1}, and we have:
\begin{equation*}
|\lambda|+|\beta|+\|\{K'_s\}\|_{A_M}\le C\,\|T\|_{A_M},
\end{equation*}
Moreover
\begin{equation*}
\Big\|\lambda\,\textup{Id}+\sum_{\twoline{M^\theta\le s<s_0}{s-\text{dyadic}}}(\beta\,\hams+K'_s)\Big\|_{\ell^2\to\ell^2}
\le C\|T\|_{\ell^2\rightarrow\ell^2}+\epsilon(M)\|T\|_{A_M}.
\end{equation*}
\end{lemma}
\begin{proof} This lemma is standard and we include the proof for the reader`s convenience.
Let $\psi$ be the  smooth symmetric cutoff function as in the  lemma \ref{1:3}, and let $s'$ be the largest dyadic integer satisfying $s'\le M^\theta/2$. We let
\begin{equation*}
\psi^{s'}(x)=\psi(\frac{x}{s'}),\text{ and }\psi^s(x)=\psi(\frac{x}{s})-\psi(\frac{2x}{s})\quad \text{for }s>s',
\end{equation*}
and thus
\begin{equation*}
\sum_{\twoline{s_0\ge s\ge s'}{s-\text{dyadic}}}\psi^s(x) = \psi(\frac{x}{s_0})=\psi_{s_0}(x),
\end{equation*}
with
\begin{equation*}
\textup{supp }\psi^{s'}\subset\{|x|\le M^\theta\}\quad\textup{supp }\psi^{s}\subset\{s/2\le|x|\le2s\},\quad s>s'.
\end{equation*}
Given an operator $T$ with kernel of the form \eqref{2}:
\begin{equation*}
T=\lambda\,\textup{Id}+\beta\,\ham+\sum_{\twoline{M^\theta\le s<\infty}{s-\text{dyadic}}} K_s,
\end{equation*}
we can write the decomposition of its kernel
\begin{equation*}
\psi_{s_0}\,\cdot\,T=\lambda\,\textup{Id}+\beta\sum_{\twoline{s_0\ge s\ge 2s'}{s-\text{dyadic}}}\psi^s\,\cdot\,\ham+\sum_{\twoline{s_0\ge s\ge s'}{s-\text{dyadic}}}\psi^s\,\cdot\,K,
\end{equation*}
where
\begin{equation*}
K=\sum_{\twoline{M^\theta\le s<\infty}{s-\text{dyadic}}} K_s.
\end{equation*}

Now we let
\begin{gather*}
\hams= \psi^s\,\cdot\,\ham,\qquad s>s',\\
\tilde{K}_s=\psi^s\,\cdot\,K,\qquad s\ge s'.
\end{gather*}
Observe, that the kernels $\tilde{K}_s$ satisfy the requirements in the definition of the algebra $A_M$, except, possibly, for the vanishing means. We let
\begin{equation*}
K'_s(x)=\tilde{K}_s(x)-\,\frac{c_s}{s}\,\psi\Big(\frac{x}{s}\Big)\,\sum_{y\in\mathbb Z}\tilde{K}_s(y),
\end{equation*}
where the constants $c_s$ have been chosen so that
\begin{equation*}
\,\frac{c_s}{s}\,\sum_{x\in\mathbb Z}\psi\Big(\frac{x}{s}\Big)=1.
\end{equation*}
Note, that the kernels $K'_s$ do have vanishing means, and satisfy all the requirements of the definition of the algebra $A_M$, with $\|\{K'_s\}\|_{A_M}$ bounded by $\|\{K_s\}\|_{A_M}$.
Now we write the decomposition of kernel $T(x)$
\begin{equation*}
\psi_{s_0}(x)\,\cdot\,T(x)=\lambda\,\textup{Id}(x)+\beta\sum_{\twoline{s_0\ge s\ge 2s'}{s-\text{dyadic}}}\hams(x)+\sum_{\twoline{s_0\ge s\ge s'}{s-\text{dyadic}}}K'_s(x)+
\end{equation*}
\begin{equation*}
+\sum_{\twoline{s_0/2\ge s\ge s'}{s-\text{dyadic}}}J_s
\Big(\,\frac{c_s}{s}\,\psi\Big(\frac{x}{s}\Big)-\,\frac{c_{2s}}{2s}\,\psi\Big(\frac{x}{2s}\Big)\Big)+
J_{s_0}\,\frac{c_{s_0}}{s_0}\,\psi\Big(\frac{x}{s_0}\Big),
\end{equation*}
where
\begin{equation*}
J_s=\sum_{\twoline{s\ge l\ge s'}{s-\text{dyadic}}}\sum_y K'_l(y)=\sum_y K(y)\psi\Big(\frac{y}{s}\Big)
\end{equation*}
and $J_{ s'/2}=0$. Let
\begin{equation*}
K''_s(x)=K'_s(x) +J_{s/2}
\Big(\,\frac{2c_{s/2}}{s}\,\psi\Big(\frac{2x}{s}\Big)-\,\frac{c_{s}}{s}\,\psi\Big(\frac{x}{s}\Big)\Big)
\end{equation*}
We will prove below that $|J_s|\le |\lambda|+ C\|T\|_{\ell^2\rightarrow\ell^2}$. This immediately imply
\begin{equation*}
T=\lambda\,\textup{Id}+\beta\sum_{\twoline{s\ge 2s'}{s-\text{dyadic}}}\hams+
\sum_{\twoline{ s\ge s'}{s-\text{dyadic}}}K''_s
\end{equation*}
in a weak sense. Moreover, by lemma \ref{1:3} applied to $\psi_{s_0}\,\cdot\,T$ and
estimate on $\lambda$ provided by lemma \ref{1:5}, the partial sums
\begin{equation*}
\lambda\,\textup{Id}+\beta\sum_{\twoline{s_0\ge s\ge 2s'}{s-\text{dyadic}}}\hams+\sum_{\twoline{s_0\ge s\ge s'}{s-\text{dyadic}}}K''_s
\end{equation*}
represents an operator with  $\ell^2\rightarrow\ell^2$ bounded by $ C\|T\|_{\ell^2\rightarrow\ell^2}+\epsilon(M)\|T\|_{A_M}$,
and  by the construction $\|K''\|_{A_M}\le C\|T\|_{A_M}$.

We will now show the required estimate for $J_s$, that is
\begin{equation*}
\Big|\sum_{y\in\mathbb Z}K(y)\,\psi\Big(\frac{y}{s}\Big)\Big|\le c\|T\|_{\ell^2\to\ell^2}+|\lambda|.
\end{equation*}
We let
\begin{equation*}
K^s=(K+\ham)\,\cdot\,\psi_s,\qquad\chi_s=\frac{1}{2s+1}\chi_{[-s,s]},
\end{equation*}
and, since the kernel $\ham$ is antysymmetric
\begin{align*}
\Big|\sum_{y\in\mathbb Z}K(y)\,\psi_s(y)\Big|^2&=\Big|\sum_{y\in\mathbb Z}K^s(y)\,\sum_{y_1\in\mathbb Z}\chi_s(y_1)\Big|^2\\
&=\Big|\sum_{y\in\mathbb Z}K^s*\chi_s(y)\Big|^2\\
&\le8\,s\,\sum_{y\in\mathbb Z}\big|K^s*\chi_s(y)\big|^2\\
&\le8\,s\,\|K^s\|_{\ell^2\to\ell^2}^2\|\chi_s\|_{\ell^2}^2\\
&\le\frac{8\,s}{2s+1}\,\|K^s\|_{\ell^2\to\ell^2}^2\\
&\le c\,\|K+\ham\|_{\ell^2\to\ell^2}^2\\
&\le2 c\|T\|_{\ell^2\to\ell^2}^2 +2|\lambda|^2.
\end{align*}
where the estimate for $\|K^s\|_{\ell^2\rightarrow\ell^2}$ follows by lemma \ref{1:3}. Now we apply
lemma \ref{1:5}.
\end{proof}
\begin{lemma}\label{1:6}
Let $0\le \varphi\in\rm{C}_c^\infty(\mathbb R)$ and  $\varphi_s=\frac{c_s}{s}\varphi(\frac{\cdot}{s})$, with constants $c_s>0$ such that $\|\varphi_s\|_1=1$. For a given $\delta>0$ and a positive dyadic integer $s$ let   $s_1$ be such that
$s^{\frac{\alpha-1+\delta}{\alpha}}\le s_1\le s$. Then for $0<\gamma\le \gamma_0(\delta)$ we have:
\begin{enumerate}[(i)]
\item $\|\varphi_{s_1}*\hams\|_{\ell^2}^2\le\frac{c}{s}$,
\item $\|\varphi_{s_1}*\hams(\,\cdot+h\,)-\varphi_{s_1}*\hams\|_{\ell^2}^2\le\frac{c}{s}\big(\frac{h}{s}\big)^{\gamma}$.
\end{enumerate}
We can take $\gamma_0(\delta)=\min\{\frac{\delta}{4\alpha},\gamma(\frac{\delta}{2})\}$, where
$\gamma(\delta)$ is defined by \eqref{G_reg}.
\end{lemma}
\begin{proof}
It suffices to prove {\it (ii)} with $|h|\le Cs$ since it implies  {\it (i)}.
For the moment, the superscript ${ }^h$ denotes the translation of a function by $h$. We have:
\begin{align*}
&<(\varphi_{s_1}^h-\varphi_{s_1})*\hams,(\varphi_{s_1}^h-\varphi_{s_1})*\hams>=<(\varphi_{s_1}^h-\varphi_{s_1})*G_s,\varphi_{s_1}^h-\varphi_{s_1}>+\\
&\qquad\qquad+\|\varphi_{s_1}^h-\varphi_{s_1}\|_{\ell^2}^2\,\frac{1}{s^{1/\alpha}}+<(\varphi_{s_1}^h-\varphi_{s_1})*E_s,(\varphi_{s_1}^h-\varphi_{s_1})>\\
&\qquad=I+II+III.
\end{align*}
In the above we have applied Lemma {\ref{war_reg}} with $\delta_l=\delta/2$ to obtain  the
decomposition : $\hams*\hams=G_s+\frac{C\delta_0}{s^{1/\alpha}}+E_s$, satisfying estimates
\eqref{G_sup},   \eqref{G_reg}. We have for $\gamma\le \gamma(\delta/2)$, where  $\gamma(\delta)$
is defined by \eqref{G_reg}:
\begin{align*}
|I|&=|<(\varphi_{s_1}^h-\varphi_{s_1})*G_s,\varphi_{s_1}^h-\varphi_{s_1}>|\\
&=<\varphi_{s_1}*(G^h_s-G_s),\varphi^h_{s_1}-\varphi_{s_1}>\\
&\le C\,\frac{1}{s}\Big(\frac{|h|}{s}\Big)^\gamma\|\varphi_{s_1}\|_{\ell^1}^2.\\
|II|&\le C\,\frac{1}{s^{1/\alpha}}\,\cdot\,\frac{1}{s_1}\cdot\big(\frac{|h|}{s_1}\big)^\gamma\\
&\le C\,\frac{1}{s^{1/\alpha}}\,\cdot\,\frac{|h|^\gamma}{s_1}\\
&\le C\,\frac{1}{s^{1/\alpha}}\,\cdot\,\frac{|h|^\gamma}{s^{1-1/\alpha+\delta/\alpha}}\\
&\le C\,\frac{1}{s}\,\big(\frac{|h|}{s}\big)^{\delta/2\alpha},
\end{align*}
for $\gamma\le\delta/2\alpha$ and  $s_1\ge s^{1-1/\alpha+\delta/\alpha}$. By H\"older regularity of
$\varphi_{s_1}$
\begin{align*}
|III|&\le C\,\big\|\varphi^h_{s_1}-\varphi_{s_1}\big\|_{\ell^\infty}\|E_s\|_{\ell^1}\\
&\le C\,\Big(\frac{|h|}{s_1}\Big)^\gamma\,\cdot\,\frac{1}{s_1}\,\cdot\,\frac{1}{s}\,\cdot\,s^{1-1/\alpha+\delta/2\alpha}\\
&\le \frac{C}{s^{1/\alpha}}\,\cdot\,\frac{1}{s_1}\,|h|^{\gamma}|s|^{\delta/2\alpha}\\
&\le\frac{C}{s^{1/\alpha}}\,\frac{1}{s^{1-1/\alpha+\delta/\alpha}}\,\cdot\, s^{\delta/2\alpha}\,|h|^\gamma\\
&\le\frac{C}{s}\,\Big(\frac{|h|}{s}\Big)^{\delta/4\alpha}\,\cdot\,s^{-\delta/4\alpha},
\end{align*}
for $\delta/4\alpha\ge\gamma$ and $s_1$ as in II.
\end{proof}

Let
\begin{equation*}
\tilde{\mathbb T_s}=\sum_{M^\theta<s'<s}\big(\tilde{\mathcal H_{s'}}+\tilde{K_{s'}'}\big),
\end{equation*}
where the kernels $\tilde{ \mathcal H_{s'}},\tilde{K_{s'}'}$ comes from the representation of
$\tilde{T}$ in the sense of Lemma \ref{1:4}.

\begin{lemma}\label{1:7}
For  $\gamma\le \gamma_0(\delta)$ and $s^{1-1/\alpha +\delta/\alpha}\le s_1 \le s$ we have
\begin{enumerate}[(i)]
\item $\|\varphi_{s_1}*\hams*\tilde{\mathbb T_s}\|_{\ell^2}^2\le \frac{C}{s}(\|\tilde T\|^2_{\ell^2\to\ell^2}+\frac{C}{M^\theta}\|\tilde T\|_A^2)$,
\item $\|\varphi_{s_1}*\hams*\tilde{\mathbb T_s}(\,\cdot+h\,)-\varphi_{s_1}*\hams*\tilde{ \mathbb T_s}\|_{\ell^2}^2\le \frac{C}{s}\Big(\frac{|h|}{s}\Big)^\gamma(\|\tilde T\|^2_{\ell^2\to\ell^2}+\frac{C}{M^\theta}\|\tilde T\|_A^2)$.
\end{enumerate}
\end{lemma}
\begin{proof} Immediate, from Lemmas
\ref{1:4} and \ref{1:6}.
\end{proof}
\begin{lemma}\label{1:8}
Let $0\le l\le s^{1-1/\alpha+\delta/\alpha}$, $s^\theta =s^{\alpha-1+\delta}\le s_1\le s$ and
 $\psi_l=\varphi_l-\varphi_{2l}$, where $\varphi_l$ has been defined in Lemma \ref{1:6}.
We have for  $\gamma\le \gamma_0(\delta)$:
\begin{enumerate}[(i)]
\item $\|\psi_l*\hams*\mathcal H_{s_1}\|_{\ell^2}^2\le\frac{C}{|s|^{1+\delta/2}}$,
\item $\|\psi_l*\hams*\mathcal H_{s_1}(\,\cdot+h\,)-\psi_l*\hams*\mathcal H_{s_1}\|_{\ell^2}^2\le\frac{C}{|s|^{1+\delta/4\alpha}}\cdot\Big(\frac{|h|}{|s|}\Big)^{\gamma}$,
\item $\|\psi_l*\hams*K_{s_1}\|_{\ell^2}^2\le\frac{C}{|s|^{1+\delta/2\alpha}}$,
\item $\|\psi_l*\hams*K_{s_1}(\,\cdot+h\,)-\psi_l*\hams*K_{s_1}\|_{\ell^2}^2\le\frac{C}{|s|^{1+\delta/4\alpha}}
    \cdot\Big(\frac{|h|}{|s|}\Big)^{\gamma}$.
\end{enumerate}
\end{lemma}
\begin{proof}
(ii) and (iv) follow from (i) and (iii), since $|h|\ge1$. We will now prove (i). We again use Lemma \ref{war_reg}
with $\delta_L=\delta/2$.
\begin{align*}
&\|\psi_l*\hams*\mathcal H_{s_1}\|_{\ell^2}^2=<\psi_l*G_s,\mathcal H_{s_1}*\mathcal H_{s_1}*\psi_l>+\\
&\qquad+<\psi_l*E_s,\mathcal H_{s_1}*\mathcal H_{s_1}*\psi_l>+\\
&\qquad+<\psi_l\cdot\frac{1}{s^{1/\alpha}},\mathcal H_{s_1}*\mathcal H_{s_1}*\psi_l>\\
&\quad =I+II+III.
\end{align*}
We estimate each part:
\allowdisplaybreaks
\begin{align*}
|I|&\le\|\psi_l*G_s\|_{\ell^\infty}\cdot\|\mathcal H_{s_1}*\mathcal H_{s_1}*\psi_l\|_{\ell^1}\\
&=C\Big(\frac{|s|^{1-1/\alpha+\delta/\alpha}}{|s|}\Big)^\gamma\cdot\frac{1}{|s|}\\
&\le\frac{C}{|s|}\cdot\frac{1}{|s|^{\delta_1}}.\\
|III|&=|<\psi_l*\mathcal H_{s_1},\mathcal H_{s_1}*\psi_l>|\cdot\frac{1}{s^{1/\alpha}}\\
&\le\|\mathcal H_{s_1}\|_{\ell^2}^2\,\|\psi_l\|_{\ell^1}^2\cdot\frac{1}{s^{1/\alpha}}\\
&\le\frac{C}{s_1^{1/\alpha}}\cdot\frac{1}{s^{1/\alpha}}\\
&\le\frac{1}{s^{1/\alpha}}\cdot\frac{1}{s^{1-1/\alpha+\delta/\alpha}}\\
&\le\frac{1}{s^{1+\delta/\alpha}}.\\
|II|&=|<E_s,\mathcal H_{s_1}*\mathcal H_{s_1}*\psi_l*\psi_l>|\\
&\le|<E_s*\mathcal H_{s_1},\mathcal H_{s_1}*\psi_l*\psi_l>|\\
&\le \|E_s\|_{\ell^1}\cdot\|\mathcal H_{s_1}\|_{\ell^2}^2\\
&\le\frac{s^{1-1/\alpha+\delta/2\alpha}}{s}\cdot\frac{1}{s_1^{1/\alpha}}\\
&\le\frac{s^{1-1/\alpha+\delta/2\alpha}}{s\cdot s^{1-1/\alpha+\delta/\alpha}}\\
&\le\frac{1}{s^{1+\delta/2\alpha}}.
\end{align*}
The estimates of $|II|$ is very crude but it suffices for our purposes.
The proof of (iii) is identical.
\end{proof}
\begin{lemma}\label{1:9}
We notice:
\begin{gather}
\|\hams*\tilde{\mathbb T}_s\|_{\ell^2}^2\le\frac{C}{s}\,
\Big(\|\tilde T\|_{\ell^2}^2+\|\tilde T\|_A\cdot(\frac{1}{s^{\delta/4\alpha}}+\epsilon(s))\Big),\label{1:11}\\
\|\hams* \tilde{\mathbb T}_s(\,\cdot+h\,)-\hams*\tilde{\mathbb
T}_s\|_{\ell^2}^2\le\frac{C}{|s|}\,\Big(\frac{|h|}{|s|}\Big)^\gamma\,\Big(\|\tilde
T\|_{\ell^2}^2+\|\tilde T\|_A\cdot(\frac{1}{s^{\delta/4\alpha}}+\epsilon(s))\Big),\label{1:12}
\end{gather}
where ${\mathbb T_s}, \tilde{\mathbb T_s}$ has been defined before Lemma \ref{1:7}.
\end{lemma}
\begin{proof} It is a corollary of Lemmas \ref{1:7} and \ref{1:8}.
Let $s_1=s^{1-1/\alpha +\delta/\alpha}$, $\varphi \in C_c^\infty(-\frac 12, \frac12)$ and
$\varphi_{s_1}$ $\psi_l$  be as in Lemma \ref{1:8}. Then $\delta_0=
\varphi_{s_1}+\sum_{\twoline{l=1}{l-\text{dyadic}}}^{\frac12 s_1}\psi_l$. The conclusion of the
Lemma  follows directly from the formula:
\begin{equation}\label{oper_id}
\hams*\tilde{\mathbb T}_s=\varphi_{s_1}*\hams*\tilde{\mathbb T}_s+
\sum_{\twoline{M^\theta \le s'\le s}{{s'-\text{dyadic}}}}
\sum_{\twoline{l=1}{l-\text{dyadic}}}^{\frac12 s_1}\psi_l*\hams*(\tilde{\mathcal H}_{s'}+\tilde K_{s'}),
\end{equation}

Since the kernels $\mathbb T_s,\tilde{\mathbb T}_s$ are supported in $[-Cs,Cs]$ for some constant $C$, from
Lemma \ref{1:7} we conclude that
\begin{equation*}
\varphi_{s_1}*\hams*\tilde{\mathbb T}_s
\end{equation*}
satisfies \eqref{1:11} and \eqref{1:12}, that is the $(i)_{s_2}-(iv)_{s_2}$ of the definition
\ref{1} for some $s_2=Cs$ and with the constant $D_{s_2}\le C\|\tilde{\mathbb
T}_s\|_{\ell^2\to\ell^2} \le C\|\tilde{ T}\|_{\ell^2\to\ell^2}+C\epsilon(s)\|\tilde{ T}\|_{A_M}$.
Since $s^{\theta}\le s' \le s$, each of the kernels
\begin{equation}\label{range_s}
\psi_l*\hams*(\tilde{\mathcal H}_{s'}+\tilde K_{s'})
\end{equation}
by Lemma \ref{1:8} satisfies \eqref{1:11} and \eqref{1:12}, that is the  $(i)_{s_2}-(iv)_{s_2}$ of the definition \ref{1}  with $s_2=Cs$
 and $D_{s_2}\le Cs^{-\delta/8\alpha}\|T\|_{A_M}\le CM^{-\frac{\delta(\alpha-1+\delta)}{8\alpha}}\|T\|_{A_M}$.
Since the number of summands in \eqref{oper_id} is at most  $ C(\log M)^2$, the lemma follows.
\end{proof}
\begin{lemma}\label{1:10}
We have:
\begin{equation*}
\|T\tilde T\|_{A_M}\le C\,\big(\|T\|_{\ell^2\to\ell^2}\|\tilde T\|_{A_M}+\|T\|_{A_M}\|\tilde T\|_{\ell^2\to\ell^2}\big)+\epsilon_1(M)\|T\|_{A_M}\|\tilde T\|_{A_M},
\end{equation*}
where $\epsilon_1(M)\le CM^{-\frac{\delta(\alpha-1+\delta)}{16\alpha}}$,
and the constant $C$ does not depend on $M$.
\end{lemma}
\begin{proof}
We use the identity
\begin{align*}
T\,\tilde T&=
&\lambda\,\tilde T+\tilde\lambda\,T+\sum_s(K_s+\hams)*\tilde{\mathbb T}_s+\sum_s(\tilde K_s+\tilde\hams)*\mathbb T_{2s},
\end{align*}
($\mathbb T_s,\,\tilde{\mathbb T}_s$ as in the previous Lemma).
We apply Lemma \ref{1:9}, and obtain the estimates in the case $s\le M$. The case $s>M$ is immediate, since then $\hams$ vanish
and by $\ell^2$ boundedness of ${\mathbb T}_s,\tilde{\mathbb T}_s$, the kernels $K_s*\tilde{\mathbb T}_s$,  $\tilde K_s*{\mathbb T}_s$ satisfy conditions $(i)_{Cs}...(iv)_{Cs}$ of definition \ref{1}  with appropriate  norm controll.
\end{proof}

\section{Proof of Lemma \ref{war_reg}.}

In this section we  slightly abuse the notation and denote generic $s$ by $M$. We note that $H_M$,
introduced in \eqref{H_s_def} is supported in $[-CM^\alpha,CM^\alpha]$. Denote $G_M=H_M*H_M$. The
estimates  \eqref{G_reg} and \eqref{G_sup} on $G_M$ have been  proved in \cite{UZ}, the  estimate
\eqref{G_reg} under additional restriction  $M^{\frac{99}{100}}\le |x|,|x+u|$ and the estimate
\eqref{G_sup} for any $x\ne 0$. In what follows we will prove \eqref{G_reg} for the  remaining case
$M^{\alpha-1 +\delta_L }\le x, x+u\le M^{\frac{99}{100}}$. Then
 the new function   $\tilde G_M$  defined  on the whole
$\mathbb Z$ by $\tilde G_M(x)=G_M(x)$ for $|x|\ge M^{\alpha-1 +\delta_L }$ and
$\tilde G_M(x)=G_M([ M^{\alpha-1 +\delta_L }]) $
for
$|x|\le M^{\alpha-1 +\delta_L }$  satisfies \eqref{G_reg}.
Since $G_M(x)=G_M(-x)$, for $|x|\ge M^{\alpha-1 +\delta_L }$
we obviously have, for those $x$, $\tilde G_M(x)=G_M(x)$.
We will denote $\tilde G_M$ again by $G_M$ and define $E_M(x)$ by equation \eqref{conv_dec} with additional condition
$G_M(0)+E_M(0)=0$. Then
$E_M(x)$ obviously satisfies  \eqref{G_sup}.

We will apply the method of trigonometric polynomials and we
refer the reader to \cite{G} for all background facts.
We begin with some definitions used in the sequel.

\begin{definition*} Let $\delta>0$ be small, and $\delta_0=\frac{\delta}{100}$. We consider the partition of the interval $[0,1)$ into intervals of the form
\begin{equation*}
I_r=\Big[\frac{r}{M^{\delta_0}},\frac{r+1}{M^{\delta_0}}\Big)\subset[0,1),\qquad0\le r<M^{\delta_0}
\end{equation*}
For a number $\Delta\in [0,1)$ we will denote by $I(\Delta)$ the unique interval of the above form
such that $\Delta\in I(\Delta)$. We will write $I_r=[a(I_r),b(I_r))$ and denote by
$l(\Delta)=l(I(\Delta))=b(I(\Delta))-a(I(\Delta))$ the length of $I(\Delta)$.
\end{definition*}

Furthermore, we let $m(h,x,\Delta)$ be the unique, if it exists, non-negative solution of
\begin{equation}\label{2:8}
\big(m+h\big)^\alpha-m^\alpha =x+\Delta,
\end{equation}
where $x,h\in\mathbb N$ and $0\le\Delta<1$. Let
\begin{gather}\label{2:9}
H=\frac{x}{M^{\alpha-1}},\quad x\in\mathbb N,\quad M^{\alpha-1 +\delta_L }\le x\le
M^{\frac{99}{100}},\\
\|w\|=\inf_{k\in\mathbb Z}|k-w|,\quad w\in\mathbb R.
\end{gather}
We will consider the following condition for $(h,x,\Delta,k)$:
\begin{equation}\label{2:6}
\begin{aligned}
&\forall\  m,\quad m(h,x,a(I(\Delta)))\le m\le m(h,x,b(I(\Delta)))\ \Longrightarrow\ \\
&\qquad\Longrightarrow\ \big\|\alpha\cdot k\cdot m^{\alpha-1}\big\|\ge M^{-\delta_0/2}.
\end{aligned}
\end{equation}
\begin{lemma}\label{2:10}
If $\frac{M}{2}\le m \le 2M$ and satisfies \eqref{2:8}, and $H,\ x,\ h,\ \Delta$ as above then
\begin{equation}\label{Hrange}
 {C^{-1}}H\le h \le CH
\end{equation}
for some constant $C$ independent of $M,x,h,\Delta$. Moreover we have the following estimates:
\begin{align}
&m(h,x,b(I(\Delta)))-m(h,x,a(I(\Delta)))=c_\alpha\frac{l(I(\Delta))}{h}m(h,x,0)^{2-\alpha}\big(1+O(M^{-\delta_0})\big),\label{2:1}\\
&m(h,x,0)=\Big(\frac{x}{h\,\alpha}\Big)^\rho\big(1+O(M^{-\delta_0})\big),\label{2:2}\mbox{ where $\rho= \frac{1}{\alpha-1}$}\\
&S=\sum_{\twoline{H/C\le h\le CH}{I_r\subset[0,1)}}\varphi\Big(\frac{m(h,x,b_r)}{M}\Big)^2\,\frac{l(I_r)}{h}\,m(h,x,b_r)^{2-\alpha}=c_\alpha M^{2-\alpha} \big(1+O(M^{-\delta_0})\big),\label{2:3}
\end{align}
where the choice of $b_r\in I_r$ is arbitrary, and $\varphi\in C^\infty_c(\frac12,2)$.
\end{lemma}
\begin{proof}
The estimate \eqref{Hrange} follows immediately from the Taylor's formula. In order to prove
\eqref{2:1} we use the mean value theorem and the definition of $m(h,x,t)$:
\begin{gather}
\label{dmest}
\frac{\partial m(h,x,t)}{\partial t}=
\frac{\partial m(h,x,t)}{\partial x}=\frac{m(h,x,t)^{2-\alpha}}{\alpha(\alpha-1)h}\Big(1+O\Big(\frac{h}{M}\Big)\Big)
= O\Big(\frac{M}{x}\Big),\\
\frac{m(h,x,t)^{2-\alpha}-m(h,x,0)^{2-\alpha}}{m(h,x,0)^{2-\alpha}}=\frac{(2-\alpha)m(h,x,t_1)^{1-\alpha}\frac{\partial m(h,x,t_1)}{\partial x}}{m(h,x,0)^{2-\alpha}}=O\Big(\frac{1}{x}\Big).
\end{gather}
Hence:
\begin{align*}
&m(h,x,b(I(\Delta)))-m(h,x,a(I(\Delta)))=l(\Delta)\cdot\frac{\partial m(h,x,t)}{\partial x}=\\
&=l(\Delta)\frac{m(h,x,t)^{2-\alpha}}{\alpha(\alpha-1)h}\Big(1+O\Big(\frac{h}{M}\Big)\Big)=
l(\Delta)\,\frac{m(h,x,0)^{2-\alpha}}{\alpha(\alpha-1)h}\Big(1+O\Big(\frac{h}{M}\Big)\Big)\Big(1+O\Big(\frac{1}{x}\Big)\Big).
\end{align*}
We now prove \eqref{2:2}. Let $x_1$ be such that
\begin{equation*}
m(h,x_1,0)=\Big(\frac{x}{h\alpha}\Big)^\rho.
\end{equation*}
that is
\begin{equation*}
x_1=\Big(\Big(\frac{x}{h\alpha}\Big)^\rho+h\Big)^\alpha-\Big(\frac{x}{h\alpha}\Big)^{\rho\alpha}.
\end{equation*}
Using the Taylor's formula applied to \eqref{2:8} we obtain $|x_1-x|\le x M^{-1/100}$. We have:
\begin{align*}
&\Big|\frac{m(h,x_1,0)-m(h,x,0)}{m(h,x_1,0)}\Big|\le \frac{C}{M}\,
\frac{\partial m(h,x_1,b){|x_1-x|}}{\partial x_1}\le\\
&\qquad\le \frac{C_1}{M}\,\frac{M}{x}\,|x-x_1|\le M^{-\frac{1}{100}}.
\end{align*}
We now prove the last part, \eqref{2:3}. Using the estimate  \eqref{dmest} it is straightforward to
check that
\begin{align*}
S&=\Big(\sum_{\twoline{H/C\le h\le CH}{I_r\subset[0,1)}}\varphi\Big(\frac{m(h,x,0)}{M}\Big)^2\,\frac{m(h,x,0)^{2-\alpha}}{h}l(I_r)\Big)\Big(1+O\big(M^{-(\alpha-1+\delta)}\big)\Big)\\
&=\Big(\sum_{H/C\le h\le CH}\varphi\Big(\frac{m(h,x,0)}{M}\Big)^2\,\frac{m(h,x,0)^{2-\alpha}}{h}\Big)\Big(1+O\big(M^{-(\alpha-1+\delta)}\big)\Big).
\end{align*}
We apply \eqref{2:2} and  replace $m(h,x,0)$ by $m(h,x_1,0)$ . W get
\begin{align*}
&=\Big(\sum_{H/C\le h\le CH}\varphi\Big(\frac{1}{M}\,\Big(\frac{x}{\alpha h}\Big)^\rho\Big)^2\,\Big(\frac{x}{\alpha h}\Big)^{\rho(2-\alpha)}\,\frac{1}{h}\Big)\Big(1+O\big(M^{-(\alpha-1+\delta)}\big)\Big)\\
&=\Big(\int_0^\infty\varphi\Big(\frac{1}{M}\,\Big(\frac{x}{\alpha h}\Big)^\rho\Big)^2\,\Big(\frac{x}{\alpha h}\Big)^{\rho(2-\alpha)}\,\frac{dh}{h}\Big)\Big(1+O\big(M^{-\delta}\big)\Big).
\end{align*}
The last equality follows from \eqref{Hrange}, and the fact, that by \eqref{Hrange}
\begin{equation*}
\varphi\Big(\frac{1}{M}\,\Big(\frac{x}{\alpha h}\Big)^\rho\Big)=0\qquad\text{for $h\le C^{-1}{H}$ or $h\ge CH$},
\end{equation*}
and the Taylor's formula. Now, by the change of variables, the last integral equals to $c_\alpha M^{2-\alpha}$ and \eqref{2:3} follows.
\end{proof}
\begin{lemma}
Let $M^{\alpha-1+\delta_L}\le x\le M^{\frac{99}{100}}$. We then have:
\begin{equation*}
M^2\,H_M*H_M(x)=\sum_{\twoline{H/C\le h\le CH}{I_r\subset[0,1)}}\varphi\Big(\frac{m(h,x,a(I_r))}{M}\Big)^2\Big(\big|\mathcal J^-_{h,x,I_r}\big|+\big|\mathcal J^+_{h,x-1,I_r}\big|\Big) + Er(x),
\end{equation*}
where $\mathcal J^-_{h,x,I_r}$, and $\mathcal J^+_{h,x,I_r}$ are sets satisfying the inclusions:
\begin{align*}
\mathcal J^-_{h,x,I_r}&\supset\{m\in[m(h,x,a(I_r)),m(h,x,b(I_r)):\ \{m^\alpha\}\ge 1-a(I_r)\},\\
\mathcal J^-_{h,x,I_r}&\subset\{m\in[m(h,x,a(I_r)),m(h,x,b(I_r)):\ \{m^\alpha\}\ge 1- b(I_r)\},\\
\mathcal J^+_{h,x,I_r}&\supset\{m\in[m(h,x,a(I_r)),m(h,x,b(I_r)):\ \{m^\alpha\}\le 1-b(I_r)\},\\
\mathcal J^+_{h,x,I_r}&\subset\{m\in[m(h,x,a(I_r)),m(h,x,b(I_r)):\ \{m^\alpha\}\le 1-a(I_r)\}.
\end{align*}
Moreover, for the error function $ Er(x)$ we have $| Er(x)|\le CM^{1-\alpha}M^{2-\alpha}$ so it satisfies
conditions \eqref{G_sup} and \eqref{G_reg} required for $G$.
\end{lemma}
\begin{proof} By the definition of $ H_M$,
we have:
\begin{align*}
&M^2\,H_M* H_M(x)=\\
&\qquad=\sum_{{m_1,m_2\in\mathbb Z}}\varphi\Big(\frac{m_1}{M}\Big)
\frac{M}{m_1}\varphi\Big(\frac{m_2}{M}\Big)\frac{M}{m_2}
\delta_{\pm [m_1^\alpha]}*\delta_{\pm [m_2^\alpha]}(x)\\
&\qquad=2\sum_{{m_1,m_2\in\mathbb Z}}\tilde\varphi\Big(\frac{m_1}{M}\Big)
\tilde\varphi\Big(\frac{m_2}{M}\Big)
\delta_{[m_1^\alpha] -[m_2^\alpha]}(x)=(\dagger)
\end{align*}
where we have denoted $\tilde\varphi(t)=\rm{sgn}\,(t)|t|^{-1}\varphi(t)$, and used the fact that for $m_1>m_2$ and
$0<x\le M^{\frac{99}{100}}$ the equation $\pm [m_1^\alpha]\pm [m_2^\alpha]=x$ can be solved only when
$ [m_1^\alpha]- [m_2^\alpha]=x$ .

We now fix $h>0$ and consider solutions to the equation:
\begin{equation*}
x=[m_1^\alpha]-[m_2^\alpha], \quad m_1-m_2=h, \quad \frac{M}{2}\le m_1\le2M.
\end{equation*}
Each solution is a pair $m_1,m_2$, but it is determined uniquely by its larger component $m_1$. In the following we refer to $m_1$ as ``the solution''. The set $\mathcal J^+_{h,x,I_r}$ consists of solutions with additional condition
\begin{equation*}
m_1^\alpha-m_2^\alpha=x+\Delta,\qquad \Delta\in I_r\subset[0,1).
\end{equation*}
The complementary set, $\mathcal J^-_{h,x,I_r}$ consists of solutions with additional condition
\begin{equation*}
m_1^\alpha-m_2^\alpha=x-1+\Delta,\qquad \Delta\in I_r\subset[0,1).
\end{equation*}
It is immediate, that if $\big[(m+h)^\alpha\big]-\big[m^\alpha\big]=x$ then
\begin{equation*}
(m+h)^\alpha-m^\alpha=x+\Delta,
\end{equation*}
or
\begin{equation*}
(m+h)^\alpha-m^\alpha=x-1+\Delta,
\end{equation*}
for some $\Delta\in[0,1)$.
Hence
\begin{equation*}
\Big\{\frac{1}{2}M\le m\le 2M: (\exists\,k)\,x=[m^\alpha]-[k^\alpha]\Big\}=\bigcup_{\twoline{H/C\le h\le CH}{I_r\subset[0,1)}}\mathcal J^+_{h,x,I_r}\cupdot\mathcal J^-_{h,x,I_r}.
\end{equation*}
Hence, we have
\begin{equation*}
(\dagger) =2\sum_{I_r\subset[0,1)}\sum_{H/C\le h\le CH}\sum_{{m_1\in\mathcal J^+_{h,x,I_r}\cupdot\mathcal J^-_{h,x,I_r}}}\tilde\varphi\Big(\frac{m_1}{M}\Big)
\tilde\varphi\Big(\frac{m_2}{M}\Big)
\delta_{[m_1^\alpha] -[m_2^\alpha]}(x)=(\ddag)
\end{equation*}
Since for ${m_1\in\mathcal J^+_{h,x,\Delta}\cupdot\mathcal J^-_{h,x,\Delta}}$ we have by
\eqref{2:1} $|m_1-m(h,x, a(\Delta))|\le CM^{2-\alpha},|m_2-m(h,x,a(\Delta))|\le
CM^{2-\alpha}+C|m_2-m_1|\le CM^{2-\alpha}+CH\le CM^{2-\alpha} $, applying Taylor formula for
$\varphi$ we get
\begin{equation*}
(\ddag)=2\sum_{I_r\subset[0,1)}\sum_{H/C\le h\le CH}\tilde\varphi\Big(\frac{m(h,x,a(I_r))}{M}\Big)^2\sum_{{m_1\in\mathcal J^+_{h,x,I_r}\cupdot\mathcal J^-_{h,x,I_r}}}1 +Er(x)
\end{equation*}
where the error term $Er(x)$ satisfies
\begin{equation}
|Er|\le CM^{1-\alpha}\#\Big\{\frac{1}{2}M\le m\le 2M: (\exists\,k)\,x=[m^\alpha]-[k^\alpha]\Big\}
\le CM^{1-\alpha}M^{2-\alpha}
\end{equation}
The last inequality,  by \cite{UZ} is true for every $x\in \mathbb Z$.
The first statement of Lemma follows.

If for some $\Delta\in I(\Delta)\subset[0,1)$ we have
\begin{equation*}
(m+h)^\alpha-m^\alpha=x+\Delta,\qquad x\in\mathbb N,
\end{equation*}
and
\begin{equation*}\{m^\alpha\}\le 1-b(I(\Delta)),
\end{equation*}
then
\begin{equation*}
\big[(m+h)^\alpha\big]-\big[ m^\alpha\big]=x.
\end{equation*}
So
\begin{equation*}
\{m^\alpha\}+\{(m+h)^\alpha-m^\alpha\}\le 1-b(I(\Delta))+\Delta,
\end{equation*}
and thus
\begin{equation*}
\{(m+h)^\alpha\}=\{m^\alpha\}+\{(m+h)^\alpha-m^\alpha\}=\{m^\alpha\}+\Delta.
\end{equation*}
So,
\begin{equation*}
\big[(m+h)^\alpha\big]-\big[m^\alpha\big]=x+\Delta -\big(\{(m+h)^\alpha\}-\{m^\alpha\}\big)=x.
\end{equation*}
Analogously:
\begin{gather*}
\{m^\alpha\}\ge 1-a(I(\Delta))\ \Rightarrow\ \{m^\alpha\}+\{(m+h)^\alpha-m^\alpha\}>1\ \Rightarrow\\
\Rightarrow\ \{(m+h)^\alpha\}=\{m^\alpha\}+\Delta-1,
\end{gather*}
and then
\begin{equation*}
\big[(m+h)^\alpha\big]-\big[m^\alpha\big]=x-1.
\end{equation*}
It follows, that
\begin{equation*}
\big[(m+h)^\alpha\big]-\big[m^\alpha\big]=x\ \Rightarrow\ \{m^\alpha\}\le 1-a(I(\Delta)).
\end{equation*}
The required inclusions now follow.
\end{proof}
Let us introduce the following 4 functions. Given an interval $I_r\subset[0,1)$ let
\begin{equation*}
\chi_1=\chi_{[1-a(I_r), 1-M^{-\delta_0}]},\quad \chi_2=\chi_{[1-b(I_r), 1]}.
\end{equation*}
Also, choose a function $\varphi$, smooth, even, positive, monotone on $\mathbb R^+$, with support contained in $[-M^{-\delta_0},M^{-\delta_0}]$, and with integral 1. Extend these three functions as $1$-periodic on $\mathbb R$ ($M^{-\delta_0}<<1$), and let
\begin{equation*}
\psi_{M,I_r}^{-,-}=\chi_1*\varphi,\quad\psi_{M,I_r}^{-,+}=\chi_2*\varphi,
\end{equation*}
where the convolutions are on the torus. Using Lemma \ref{2:11} we have the following obvious estimates:
\begin{align*}
&\sum_{m(h,x,a(I_r))\le m\le m(h,x,b(I_r))}\psi_{M,I_r}^{-,-}(m^\alpha)\le\big|\mathcal J_{h,x,I_r}^-\big|,\\
&\big|\mathcal J_{h,x,I_r}^-\big|\le\sum_{m(h,x,a(I_r))\le m\le m(h,x,b(I_r))}\psi_{M,I_r}^{-,+}(m^\alpha).
\end{align*}
We now choose new
\begin{equation*}
\chi_1=\chi_{[M^{-\delta_0},1-b(I_r)]},\quad \chi_2=\chi_{[0,1-a(I_r)]},
\end{equation*}
and let
\begin{equation*}
\psi_{M,I_r}^{+,-}=\chi_1*\varphi,\quad\psi_{M,I_r}^{+,+}=\chi_2*\varphi.
\end{equation*}
In this case, we have
\begin{align*}
&\sum_{m(h,x,a(I_r))\le m\le m(h,x,b(I_r))}\psi_{M,I_r}^{+,-}(m^\alpha)\le\big|\mathcal J_{h,x,I_r}^+\big|,\\
&\big|\mathcal J_{h,x,I_r}^+\big|\le\sum_{m(h,x,a(I_r))\le m\le m(h,x,b(I_r))}\psi_{M,I_r}^{+,+}(m^\alpha).
\end{align*}
It is straightforward to see, that if $\psi$ is any one of the above introduced functions we have the estimates:
\begin{gather}
\sum_{k\in\mathbb Z}\big|\hat\psi(k)\big|\le C\,\log M,\label{fourier_1}\\
\sum_{|k|>M^{2\delta_0}}\big|\hat\psi(k)\big|\le C\,M^{-\delta_0}\label{fourier_2}.
\end{gather}
\begin{lemma}\label{2:12}
We have an estimate
\begin{multline*}
\Big|\sum_{m(h,x,a(I_r)\le m\le m(h,x,b(I_r))}\psi(m^\alpha)-(m(h,x,b(I_r))-m(h,x,a(I_r)))\int_0^1\psi(t)\,dt\Big|\le\\
\le\sum_{0<|k|\le
M^{2\delta_0}}\big|\hat\psi(k)\big|\,\big|S_k(h,x,I_r)\big|+\frac{C}{M^{\delta_0/4}}|m(h,x,b(I_r))-m(h,x,a(I_r))|,
\end{multline*}
where $\psi$ is any of the functions $\psi_{M,I_r}^{\pm}$, and
\begin{equation}\label{2:4}
\big|S_k(h,x,I_r)\big|\le\frac{1}{M^{\delta_0/4}}|m(h,x,b(I_r))-m(h,x,a(I_r))|
\end{equation}
if $(h,x,\Delta,k)$ satisfies \eqref{2:6} and always
\begin{equation}\label{2:5}
\big|S_k(h,x,I_r)\big|\le C|m(h,x,b(I_r))-m(h,x,a(I_r))|
\end{equation}
\end{lemma}
\begin{proof}
Let us denote
\begin{equation}\label{Jdef}
\mathcal J=\{m(h,x,a(I_r))\le m\le m(h,x,b(I_r))\}.
\end{equation}
We have
\begin{align*}
&\Big|\sum_{m\in\mathcal J}\psi(m^\alpha)-\sum_{m\in\mathcal J}\hat\psi(0)\Big|\le\\
&\qquad\qquad\le
\sum_{0<|k|\le M^{2\delta_0}}\big|\hat\psi(k)\big|\Big|\sum_{m\in \mathcal J}\,e^{2\pi\,i\,m^\alpha\cdot k}\Big|+|\mathcal J|\cdot\sum_{|k|>M^{2\delta_0}}\big|\hat\psi_{M,I_r}(k)\big|\\
&\qquad\qquad=I+II.
\end{align*}
It follows from \eqref{fourier_2} that $II\le|\mathcal J|\,M^{-\delta_0}$. We will estimate $I$. We have, as in the proof of Van der Corput's difference lemma, \cite{G}:
\begin{align*}
\Big|\sum_{m\in\mathcal J}e^{2\pi\,i\,m^\alpha k}\Big|&\le\frac{1}{D}\sum_{m\in\mathcal J}
\Big|\sum_{s=0}^{D-1}e^{2\pi\,i\,((m+s)^\alpha-m^\alpha)\cdot k}\Big|+C\cdot D\\
&\le\frac{1}{D}\sum_{m\in\mathcal J}\Big|\sum_{s=0}^{D-1}e^{2\pi\,i\,ks\alpha m^{\alpha-1}}\Big|+C|\mathcal J|\Big(\cdot\frac{D^2 M^{2\delta_0}}{M^{2-\alpha}}+\frac{D}{|\mathcal J|}\Big),
\end{align*}
with the second term of the last expression estimated by
$|\mathcal J|(\frac{M^{4\delta_0}}{M^{2-\alpha}}+M^{\delta_0-\frac{1}{100}})\le|\mathcal J|M^{-\delta_0}$
 if we have $D=M^{\delta_0}$.
We have used in the above the the following obvious consequence of the Taylor's formula
\begin{equation*}
e^{2\pi i ((m+s)^\alpha-m^\alpha)}=e^{2\pi i \alpha s\,m^{\alpha-1}}+O\Big(\frac{s^2\, k}{m^{2-\alpha}}\Big).
\end{equation*}
We continue the original estimate:
\begin{equation*}
\le\frac{1}{D}\sum_{m\in\mathcal J}\min\Big\{D,\frac{2}{\|\alpha k m^{\alpha-1}\|}\Big\}+\frac{C|\mathcal J|}{M^{\delta_0}}.
\end{equation*}
Now, if $(h,x,\Delta,k)$ satisfies the \eqref{2:6} condition, then
\begin{equation*}
\frac{1}{D}\sum_{m\in\mathcal J}\min\Big\{D,\frac{2}{\|\alpha k m^{\alpha-1}\|}\Big\}\le M^{-\delta_0/2}|\mathcal J|.
\end{equation*}

\end{proof}
\begin{lemma}\label{2:11}
Assume $|k|\le M^{2\delta_0}$.
We have the estimates
\begin{align*}
&\sum_{1/C\,H\le h\le C\,H}\big|S_k(h,x,I_r)\big|\le\\
&\qquad\le\frac{C\,H}{M^{\delta_0/4}}\,|m(h,x,b(I_r))-m(h,x,a(I_r))|\\
&\qquad\le C\,l(I_r)\,M^{2-\alpha-\delta_0/4}.
\end{align*}
\end{lemma}
\begin{proof}
The last inequality is an obvious consequence of \eqref{2:1}. Based on \eqref{2:4} and \eqref{2:5}
it is enough to prove the estimate
\begin{equation*}
\#\{h:\ (h,x,\Delta,k)\text{ does not satisfy \eqref{2:6}}\}\le CH\,M^{-\delta_0/4}.
\end{equation*}
To do so, let us momentarily fix $h,x,\Delta,k$ which do not satisfy \eqref{2:6}, and thus there
exists $m\in \mathcal J$ such, that
\begin{equation*}
\big\|\alpha\,k\,m^{\alpha-1}\big\|<M^{-\frac{\delta_0}{2}}.
\end{equation*}
Let $|k|\le M^{2\delta_0}$. We will show the estimate
\begin{equation*}
\alpha\,k\,m^{\alpha-1}=\frac{kx}{h}+O\big(M^{-\frac{\delta}{2}}\big),
\end{equation*}
Since $m\in J$, it satisfies the equation
\begin{equation*}
(m+h)^\alpha-m^\alpha=x+\Delta,\qquad a(I(\Delta))\le\Delta<b(I(\Delta)),
\end{equation*}
and by the mean-value theorem
\begin{gather*}
\alpha\,h\,m^{\alpha-1}=x+\Delta+O\Big(\frac{h^2\,M^\alpha}{M^2}\Big),\\
\end{gather*}
By \eqref{2:9} we have $M^\delta\le H \le M^{99/100}$
and consequently since $|k|\le M^{2\delta_0}$ and  $2\delta_0<\delta/2$
\begin{gather*}
\alpha\,k\,m^{\alpha-1}=\frac{k\,x}{h}+O\big(M^{-\delta/2}\big)
\end{gather*}
We have
\begin{gather*}
\Big\|\frac{k\,x}{h}\Big\|\le\big\|\alpha\,k\,m^{\alpha-1}\big\|+{M^{-\delta/2}} \le 2M^{-\delta_0/2},\\
\end{gather*}
Now, let $w\in\mathbb N$ be the integer approximation of $\frac{kx}{h}$, thus
\begin{equation*}
\frac{kx}{h}=w+e,\qquad|e|\le 2\,M^{-\delta_0/2}.
\end{equation*}
We now assume that we have at least  $H\,M^{-\delta_0/4}$ different  $h_i$'s, with false \eqref{2:6}. Thus, each of these $h_i$'s satisfies
\begin{equation}\label{2:7}
k\,x=h_i\,w_i+e_i\,h_i,
\end{equation}
and since $kx$ and $h_iw_i$ are integers, so are $e_ih_i$, and
\begin{equation*}
|e_i\,h_i|\le2 H\,M^{-\delta_0/2}.
\end{equation*}
Now, for given number $z$ with $|z|\le 2 HM^{-\delta_0/2}$ we consider the set
\begin{equation*}
\mathcal A_z=\{h_i: kx=h_iw_i+z\}.
\end{equation*}
If for each $z$ the number of elements of $\mathcal A_z$ is $<\frac{1}{2}M^{\delta_0/4}$, that the total number of $h_i$'s satisfying \eqref{2:7}
would be $<\frac{1}{2}M^{\delta_0/4}\cdot2HM^{-\delta_0/2}=HM^{-\delta_0/4}$, which is a contradiction. Thus, there must be a $z$, for which
\begin{equation}\label{divlarge}
\#\{h_i:kx=h_iw_i+z\}\ge \frac{1}{2}\,M^{\delta_0/4}.
\end{equation}
Now, since $|z|\le \frac{Cx}{M^{\alpha-1}}$, $k\ne 0$ we have   $0\neq |kx-z|\le M^{\delta_0/2+1}$ and by \eqref{divlarge}
$kx-z$ has at least $M^{\delta_0/4}$ divisors, which is impossible by a well known estimate on the number of divisors.
\end{proof}
\begin{corollary}
We have
\begin{equation*}
\sum_{I_r}\sum_{h\sim H}\Big|\Big|\mathcal J^+_{h,x,I_r}\Big|+\Big|\mathcal J^-_{h,x,I_r}\Big|-(m(h,x,b(I_r))-m(h,x,a(I_r)))\Big|
\le C\, M^{2-\alpha-\delta_0/4}.
\end{equation*}
\begin{equation*}
\sum_{I_r}\sum_{h\sim H}\varphi\Big(\frac{m(h,x,a(I_r))}{M}\Big)^{2}
\Big(\Big|\mathcal J^+_{h,x,I_r}\Big|+\Big|\mathcal J^-_{h,x,I_r}\Big|\Big)=S +O\Big( M^{2-\alpha-\delta_0/4}\Big)
\end{equation*}
where $S$ is defined by \eqref{2:3}.
\end{corollary}
\begin{proof} The first formula  is an immediate consequence of Lemmas \eqref{2:11} and \eqref{2:12}. For the second formula we apply \eqref{2:1} and the first part.
\end{proof}

\section{A counterexample}

In this section we prove the theorem \ref{4_ex}. Fix $1<\alpha< 1+\frac1{1000}$, $0<\delta \le
\frac{(\alpha-1)^2}{\alpha}$ and $\kappa =c\delta$, where $c$ will be specified later. Let
$\{M_l\}_l$ be sequence of integers satisfying $10M_l\le  M_{l+1}^{\alpha-1-1.1\delta}$, $\varphi
\in C_c^\infty(1,2)$ real valued. We put $\varphi_s=\varphi$ if for some $l$ we have (recall $s$ is
dyadic ) $s^{}\in U_-=[ M_l^{\alpha-1-1.1\delta},M_l^{\alpha-1-\delta}]$ or $s^{}\in U_+=[
M_l^{1-0.1\kappa}, M_l^{}]$ and $\varphi_s=0$ otherwise. We will consider Hilbert transform
$\hama=\sum_{\twoline{M^{\alpha-1-1.1\delta}\le s \le M}{s -\text{dyadic}}}H_{s}$ (we use more
convenient $H_s$ instead of $\hams$ ) corresponding to this  sequence $\{\varphi_s\}$ and
$\theta=\alpha-1-1.1\delta$.

Fix $l$ and denote $M=M_l$. By \eqref{hilb_def1}, $\hama$  contains two large blocks ${\mathbb
H}_+,{\mathbb H}_-$ corresponding to summation indices in
 $ U_+, U_-$
respectively. For  $P=M^{\alpha(\alpha-1-\delta)}$ and an integer $j$ satisfying, for $C$
sufficiently large, $\frac1C M^{\alpha(2+0.9\delta-\alpha)} \le j\le CM^{\alpha(2+\delta-\alpha)}$,
let $I_j=[(j-1)P,(j+1)P]$. Consider $A_j$, the set of $n\in U_-$ such that for some $x\in I_j$ the
equation
\begin{equation}\label{11}
[m^\alpha]\pm[n^\alpha]=x
\end{equation}
has  more than 1 solution (a pair $m,n$, with $m\in U_+$ and  $n\in U_-$), we allow the different
choice of $\pm$ signs for different solutions. Let $m_1$ and $m_2$ satisfy \eqref{11} possibly with
different $x_1,x_2\in I_j$ and $n_1,n_2\in U_-$. We define
 $h=m_1-m_2$  and estimate using $m_1,m_2\in U_+$ and the Taylor's formula
\begin{equation*}
|m_1^\alpha-m_2^\alpha|\le P\ \Rightarrow\ h M^{(1-0.1\delta)(\alpha-1)}\le CM^{\alpha(\alpha-1-\delta)}
\end{equation*}
Let  $H=\frac{CM^{\alpha(\alpha-1-\delta)}}{M^{(1-0.1\delta)(\alpha-1)}}$, hence $|h|\le H$, that
is $m_1,m_2$ are contained in the interval of length $H$ containing some $m_0$ satisfying
\eqref{11}.
 If $n_1\in A_j$ then for some $n_2\neq n_1$ we have two pairs $m_1,n_1$ and $m_2,n_2$ satisfying \eqref{11}.
 In what follows we assume that the $\pm$ signs corresponding to both pairs are minus.
 By \eqref{11} we obtain
\begin{equation}\label{non_uni}
[n_1^\alpha]-[n_2^\alpha]=[m_1^\alpha]-[m_2^\alpha]=[m_1^\alpha-m_2^\alpha]+\Delta,\qquad\Delta\in\{-1,0,1\}.
\end{equation}
 We have:
\begin{align*}
&m_1^\alpha-m_2^\alpha=m_1^\alpha-m_0^\alpha+m_0^\alpha-m_2^\alpha\\
&\qquad=\alpha h_1 m_0^{\alpha-1}-\alpha h_2 m_0^{\alpha-1}+O(H^2\,M^{\alpha-2}),\quad H^2\,M^{\alpha-2}\le1.
\end{align*}
From this:
\begin{gather}\label{fin_numb}
[m_1^\alpha-m_2^\alpha]+\Delta=[\alpha\,(h_1-h_2)\,m_0^{\alpha-1}] +\Delta_1\\
\Delta_1\in\{-2,-1,0,1,2\},\quad -H \le h_1,h_2\le H.
\end{gather}

There  are at most  $5(4H+1)$ different numbers represented by right hand side of \eqref{fin_numb}.
By  lemma \ref{war_reg}, the number of solutions to
\begin{equation*}
[n_1^\alpha]\pm[n_2^\alpha]=k,\qquad 0< n_1,n_2\le M^{\alpha-1-\delta}
\end{equation*}
is at most $ CM^{(\alpha-1-\delta)(2-\alpha)}$. Thus the number of pairs $(n_1,n_2)$ with $n_1,m_1$
and $n_2,m_2$ satisfying \eqref{non_uni}, that is \eqref{11} for the same $x$, does not exceed
\begin{equation*}
M^{(\alpha-1-\delta)(2-\alpha)}\cdot 21\,H\le
C\cdot M^{\alpha-1-1.9\delta}.
\end{equation*}
The case of other choices of $\pm$ signs follows exactly the same way.
So we obtained $|A_j|\le M^{\alpha-1-1.9\delta}$.

Let $x$ be of the form
\begin{equation}\label{12}
x=[m^\alpha]\pm [n^\alpha], \qquad n\notin A_j \cup  A_{j-1} \cup  A_{j+1}, [m^\alpha]\in I_j
\end{equation}
Then one can easily verify, that $x\in  I_j \cup  I_{j-1} \cup  I_{j+1}$. We infer that the representation \eqref{12}
 is unique, and it remains unique if we drop the assumption $[m^\alpha]\in I_j$ (we remark that if $n\le M^{ \frac{\alpha-1-1.1\delta}{\alpha}}$ than this statement is immediate and do not require an argument above ).  In particular
 for $x,m,n$ related by \eqref{12}
\begin{equation*}
|\mathbb H_+*\mathbb H_-(x)|\ge \frac{1}{m\cdot n},
\end{equation*}
\begin{equation}\label{Hminus}
\mathbb H_-*\mathbb H_-(x)=0,
\end{equation}
Thus (we leave  the proof  for  the reader)
\begin{equation}\label{wrong_H}
\|\mathbb H_+*\mathbb H_-\|_{\ell^p}\ge
C\Big(\frac{\delta\kappa}{100}\Big)^\frac1p(\log M)^2,\qquad p=1+\frac{1}{\log M}.
\end{equation}
We will show the estimate
\begin{equation}\label{Hplus}
\|\mathbb H_+*\mathbb H_+\|_{\ell^p}\le C \kappa^\frac2p (\log M)^2
\end{equation}
where $p$ is as in \eqref{wrong_H}. We have $\mathbb H_+*\mathbb
H_+=\sum_{\twoline{M^{1-0.1\kappa}\le s_1,s_2\le M}{s_1,s_2-\text{dyadic}}}H_{s_1}*H_{s_1}$. Since
this expression contains at most $C\kappa^2(\log M)^2$ summands, it suffices to prove that $\|
H_{s_1}* H_{s_2}\|_{\ell^p}\le C $. Assume $s_1\ge s_2$. Since $ H_{s_1}* H_{s_2}$ is supported in
$[-Cs_1^{\alpha}, Cs_1^{\alpha}]$, by Cauchy-Schwartz, it suffices to have $\| H_{s_1}*
H_{s_2}\|_{\ell^2}^2\le Cs_1^{-\alpha}$. We have $\| H_{s_1}* H_{s_2}\|_{\ell^2}^2=\left< H_{s_1}*
H_{s_1}, H_{s_2}* H_{s_2}\right>\le C(\frac{1}{s_1s_2}+\frac{s_2^\alpha}{s_1^\alpha s_2^\alpha})$
where, since $ H_{s_2}* H_{s_2}$ is supported in $[-Cs_2^{\alpha}, Cs_2^{\alpha}]$, the last
estimate follows from the lemma \ref{war_reg}. Fix sufficiently small $c>0$ and $\kappa=c\delta$.
From the \eqref{wrong_H}, \eqref{Hplus} and  \eqref{Hminus} we infer that the estimate
\begin{equation*}
\|(\mathbb H_++\mathbb H_-)*(\mathbb H_++\mathbb H_-)\|_{\ell^p}\le\frac{C}{p-1}.
\end{equation*}
cannot hold uniformly with $M$ and $p>1$. By the definition, $\hama$ is antysymmetric
with operator $\ell^2\rightarrow \ell^2$ norm controlled independently of $M$, so it has purely imaginary spectrum
contained in some fixed interval $D\subset i\mathbb R$. Let $\Gamma$ be a contour in $\mathbb C$ enclosing $D$.
Then we have $\|(\lambda I +\hama)^{-1}\|_{ \ell^2 \rightarrow \ell^{2}} \le C$.
Now, if we have $\|(\lambda I +\hama)^{-1}\|_{ {\ell}^1 \rightarrow {\ell}^{1,\infty}} \le C$,
uniformly for $M$ and $\lambda\in \Gamma$,
we should have
$\|(\lambda I +\hama)^{-1}\|_{ \ell^p \rightarrow \ell^{p}} \le \frac{C}{p-1}$. The formula
$\hama^2=\frac{-1}{2\pi i}\oint_\Gamma \lambda^2  (\lambda I +\hama)^{-1}d\lambda$ implies that the estimate
$\|\hama^2\|_{ \ell^p \rightarrow \ell^{p}} \le \frac{C}{p-1}$ holds uniformly in $M$. A contradiction.

{\bf Remark.}
Now we return to a particular case of the result \cite{PZ} announced in Remark (vi) of Section 2. We sketch
the proof of the following fact:  for $\lambda$ fixed and $|\lambda|$ sufficiently large, the operators
$(\lambda+\hama)^{-1}$ are not of weak type (1,1) uniformly in $M$. We will remove large $|\lambda|$
requirement in \cite{PZ}.

Recall, that we have
\begin{equation*}
\hama=\hamp+\hamm
\end{equation*}
with both components comprised of summands with indices in $U_+$ and $U_-$ respectively.
\begin{lemma}\label{lem1}
We have, for $l\ge2$,
\begin{equation*}
\hamp^l=\sum_{s\ge M^{\alpha(1-0.1\kappa)}}{\rm K}^+_{s,l},\qquad\hamm^l=\sum_{s\ge M^{(\alpha-1-1.1\delta)\alpha}}{\rm K}^-_{s,l}.
\end{equation*}
The kernels ${\rm K}^+_{s,l},{\rm K}^-_{s,l}$ satisfy conditions $(i)_s \dots (iv)_s$ with the
constant $|D_s|\le C_0^l$, where $C_0$ is some universal constant.
\end{lemma}
\begin{proof} Corollary of Lemma \ref{1:8}
\end{proof}
\begin{lemma}\label{lem222}
Let $k\ge3$ and let us consider $\hamp^{k-1}\hamm$ and $\hamm^{k-1}\hamp$. Then, for $p>1$
\begin{equation*}
\|\hamp^{k-1}\hamm\|_{\ell^p\to\ell^p}\le\frac{C_0^{k-1}}{(p-1)^2},
\end{equation*}
and similar estimate for $\hamm^{k-1}\hamp$.
\end{lemma}
\begin{proof}It is
immediate corollary of Lemma \ref{lem1}.
\end{proof}
\begin{corollary} For $\lambda$ sufficiently large and fixed, we have, for all $M$ sufficiently large, $p=1+\frac{1}{\log M}$
\begin{equation*}
\|(\lambda-\hama)^{-1}\|_{\ell^p\to\ell^p}\ge\frac{C}{(p-1)^2|\lambda|^3}.
\end{equation*}
\end{corollary}
\begin{proof}
\begin{gather*}
(\lambda-\hama)^{-1}=\sum_{k=0}^\infty\frac{\hama^k}{\lambda^{k+1}},\\
\hama^k=(\hamp+\hamm)^k=\sum_{j=0}^k\binom{k}{j}\hamp^{k-j}\hamm^j.
\end{gather*}
Using Lemmas \ref{lem1} and  \ref{lem222}, for $p>1$ we have:
\begin{equation*}
\left\|\sum_{k=3}^\infty\frac{\hama^k}{\lambda^{k+1}}\right\|_{\ell^p\to\ell^p}\le\frac{C}{|\lambda|^4(p-1)^2},
\end{equation*}
moreover, by independent of $M$ near $\ell^1$ estimate  $\|\hama\|_{\ell^p\rightarrow\ell^p}\le
C(p-1)^{-1}$,
\begin{equation*}
\left\|\frac{\hama^2}{\lambda^3}+\frac{\hama}{\lambda^2}+\frac{\delta_0}{\lambda}\right\|_{\ell^p\to\ell^p}\ge
\frac{1}{|\lambda|^3}\|\hama^2\|_{\ell^p\to\ell^p}-\frac{C}{|\lambda|^2(p-1)}-\frac{1}{|\lambda|}.
\end{equation*}
For the estimate $\|\hama\|_{\ell^p\rightarrow\ell^p }\le C(p-1)^{-1}$
  one does not need the  weak type $(1,1)$ estimates on  $\hama$. Classical interpolation argument
 based on the Fourier transform estimates of $\hama$ (\cite{D}),
 produce constant $C$ independent on $M,\theta$. We leave the details for the interested reader.

Thus, for $\lambda$ and $M$ sufficiently large, $p=1+\frac{1}{\log M}$ we have
\begin{equation*}
\|(\lambda-\hama)^{-1}\|_{\ell^p\to\ell^p}\ge\frac{C}{(p-1)^2|\lambda|^3},
\end{equation*}
as in the proof of Theorem \ref{4_ex}.
\end{proof}

\end{document}